\documentclass[12pt,bezier]{article}
\usepackage{times}
\usepackage{booktabs}
\usepackage{pifont}
\usepackage{floatrow}
\usepackage{caption}
\usepackage{mathrsfs}
\usepackage[fleqn]{amsmath}
\usepackage{amsfonts,amsthm,amssymb,mathrsfs,bbding}
\usepackage{txfonts}
\usepackage{graphics,multicol}
\usepackage{graphicx}
\usepackage{color}
\usepackage{amssymb}
\usepackage{caption}
\usepackage{cite}
\usepackage{latexsym,bm}
\usepackage{indentfirst}
\usepackage{color}
\usepackage[colorlinks=true,anchorcolor=blue,filecolor=blue,linkcolor=blue,urlcolor=blue,citecolor=blue]{hyperref}
\usepackage{extarrows}
\usepackage{cite}
\usepackage{latexsym,bm}
\usepackage{mathtools}
\evensidemargin=\oddsidemargin\topmargin=-1.5cm

\newtheorem{thm}{Theorem}[section]

\newtheorem{claim}{Claim}

\newtheorem{lem}{Lemma}[section]
\newtheorem{cor}{Corollary}[section]
\newtheorem{pro}{Proposition}[section]

\theoremstyle{definition}

\addtocounter{section}{0}

\begin{document}
\title{Spectral radius, fractional $[a,b]$-factor and ID-factor-critical graphs\footnote{Supported by National Natural Science Foundation of China
(Nos. 11971445 and 12171440),
Henan Natural Science Foundation (No. 202300410377) and
Research Program of Science and Technology at Universities of Inner Mongolia Autonomous Region (No. NJZY22280).}}
\author{{\bf Ao Fan$^{a}$}, {\bf Ruifang Liu$^{a}$}\thanks{Corresponding author.
E-mail addresses: rfliu@zzu.edu.cn, fanaozzu@163.com, aoguoyan@163.com.}, {\bf Guoyan Ao$^{a, b}$}\\
{\footnotesize $^a$ School of Mathematics and Statistics, Zhengzhou University, Zhengzhou, Henan 450001, China} \\
{\footnotesize $^b$ School of Mathematics and Physics, Hulunbuir University, Hailar, Inner Mongolia 021008, China}}

\date{}
\maketitle
{\flushleft\large\bf Abstract}
Let $G$ be a graph and $h: E(G)\rightarrow [0,1]$ be a function.
For any two positive integers $a$ and $b$ with $a\leq b$, a fractional $[a,b]$-factor of $G$ with the indicator function $h$ is
a spanning subgraph with vertex set $V(G)$ and edge set $E_h$ such that $a\leq\sum_{e\in E_{G}(v)}h(e)\leq b$ for any vertex $v\in V(G)$,
where $E_h = \{e\in E(G)|h(e)>0\}$ and $E_{G}(v)=\{e\in E(G)| e~\mbox{is incident with}~v~\mbox{in}~G\}$.
A graph $G$ is ID-factor-critical if for every independent set $I$ of $G$ whose size has the same parity as $|V(G)|$, $G-I$ has a perfect matching.
In this paper, we present a tight sufficient condition based on the spectral radius for a graph to contain a fractional $[a,b]$-factor,
which extends the result of Wei and Zhang [Discrete Math. 346 (2023) 113269].
Furthermore, we also prove a tight sufficient condition in terms of the spectral radius for a graph with minimum degree $\delta$ to be ID-factor-critical.

\begin{flushleft}

\end{flushleft}

\section{Introduction}

Let $G$ be a finite, undirected and simple graph with vertex set $V(G)$ and edge set $E(G)$.
The order and size of $G$ are denoted by $|V(G)|=n$ and $|E(G)|=e(G)$, respectively.
We denote by $\delta(G)$, $i(G)$ and $o(G)$ the minimum degree,
the number of isolated vertices and the number of odd components of $G$, respectively.
We use $K_{n}$ and $I_{n}$ to denote the complete graph of order $n$ and the complement of $K_{n}$.
For a vertex subset $S$ of $G$, let $G[S]$ be the subgraph of $G$ induced by $S.$
Let $G_1$ and $G_2$ be two vertex-disjoint graphs.
We denote by $G_{1}+G_{2}$ the disjoint union of $G_1$ and $G_2$.
The join $G_1\vee G_2$ is the graph obtained from $G_{1}+G_{2}$ by adding all possible edges between $V(G_1)$ and $V(G_2)$.
For undefined terms and notions, one can refer to \cite{Bondy2008}.

Given a graph $G$ of order $n$, the {\it adjacency matrix} of $G$ is the 0-1 matrix $A(G)=(a_{ij})_{n\times n}$ indexed
by the vertex set $V(G)$ of $G$, where $a_{ij}=1$ when $v_{i}$ and $v_{j}$ are adjacent and $a_{ij}=0$ otherwise.
The eigenvalues of $A(G)$ are also called the eigenvalues of $G$.
Note that $A(G)$ is a real nonnegative symmetric matrix. Hence its eigenvalues are real,
which can be arranged in non-increasing order as $\lambda_{1}(G)\geq \lambda_{2}(G) \geq \cdots \geq \lambda_{n}(G).$
The largest eigenvalue of $A(G)$, denoted by $\rho(G)$, is called the {\it spectral radius} of $G$.

Let $g$ and $f$ be two integer-valued functions defined on $V(G)$ such that $0\leq g(v)\leq f(v)$ for each vertex $v$ in $V(G)$.
A {\it$(g,f)$-factor} of $G$ is a spanning subgraph $F$ of $G$ satisfying $g(v)\leq d_F(v)\leq f(v)$ for any vertex $v$ in $V(G)$.
Let $a$ and $b$ be two positive integers with $a\leq b$.
A $(g,f)$-factor is called an {\it $[a,b]$-factor} if $g(v)\equiv a$ and $f(v)\equiv b$ for any $v\in V(G)$.
An $[a,b]$-factor is called a {\it $1$-factor} (also called a {\it perfect matching}) if $a=b=1$.

Let $h : E(G)\rightarrow [0,1]$ be a function and $E_{G}(v)=\{e\in E(G)| e~\mbox{is incident with}~v~\mbox{in}~G\}$.
If $g(v)\leq\sum_{e\in E_{G}(v)}h(e)\leq f(v)$ holds for any vertex $v\in V(G)$,
then we call a subgraph $F$ with vertex set $V(G)$ and edge set $E_h$ a {\it fractional $(g,f)$-factor} of $G$ with indicator function $h$,
where $E_h = \{e\in E(G)|h(e)>0\}$.
A fractional $(g,f)$-factor is called a {\it fractional $[a,b]$-factor} if $g(v)\equiv a$ and $f(v)\equiv b$.
In particular, for a positive integer $k$, a fractional $[k, k]$-factor of a graph $G$ is called a {\it fractional $k$-factor} of $G$.
A fractional $1$-factor is also called a fractional perfect matching.
Note that if $G$ contains a $(g,f)$-factor, then it also contains a fractional $(g,f)$-factor.
However, if $G$ has a fractional $(g,f)$-factor, $G$ may not have a $(g,f)$-factor.

We start with the following well-known fractional $(g,f)$-factor theorem.

\begin{thm}[Anstee\cite{Anstee1990}]\label{thm1}
Let $G$ be a graph and $g,f: V(G)\rightarrow Z^+$ be two integer functions such that $g(v)\leq f(v)$ for all $v\in V(G)$.
Then $G$ has a fractional $(g,f)$-factor if and only if for any subset $S\subseteq V(G)$, we have
$$f(S)-g(T)+\sum_{v\in T}d_{G-S}(v)\geq0,$$
where $T=\{v|v\in V(G)-S~\mbox{and}~d_{G-S}(v)<g(v)\}$.
\end{thm}

If $g(v)\equiv a$ and $f(v)\equiv b$, then by Theorem \ref{thm1}, we obtain the following result.

\begin{cor}\label{cor1}
Let $G$ be a graph and let $a$, $b$ be two positive integers with $a\leq b$.
Then $G$ has a fractional $[a,b]$-factor if and only if for any subset $S\subseteq V(G)$, we have
$$b|S|-a|T|+\sum_{v\in T}d_{G-S}(v)\geq0,$$
where $T=\{v|v\in V(G)-S~\mbox{and}~d_{G-S}(v)<a\}$.
\end{cor}

There are many sufficient conditions which can assure a graph to have a fractional $[a,b]$-factors
(see for example, \cite{Lu2013, Liu2001, Liu2008}).
Cho, Hyun, O and Park \cite{Cho2021} posed the spectral version conjecture for the existence of $[a,b]$-factors in graphs.
Fan, Lin and Lu \cite{Fan2022} proved that the conjecture holds for the case $n\geq 3a+b-1$.
Very recently, Wei and Zhang \cite{Wei2023} confirmed the full conjecture.

\begin{thm}[Wei and Zhang \cite{Wei2023}]\label{th23}
Let $a, b$ be two positive integers with $a\leq b$, and let $G$ be a graph of order $n\geq a+1$.
If $\rho(G)>\rho(K_{a-1}\vee(K_{n-a}+K_1))$ and $na\equiv 0 ~(\rm{mod}~2)$ when $a=b$, then $G$ has an $[a,b]$-factor.
\end{thm}

It is well known that if $G$ contains an $[a,b]$-factor, then it contains a fractional $[a,b]$-factor.
Inspired by the work of Wei and Zhang \cite{Wei2023},
we obtain a tight sufficient condition in terms of the spectral radius for a graph to contain a fractional $[a,b]$-factor.

\begin{thm}\label{main1}
Let $a$, $b$ be two positive integers with $a\leq b$, and let $G$ be a graph of order $n\geq a+1$.
If $$\rho(G)\geq\rho(K_{a-1}\vee(K_{n-a}+K_1))$$ and $na\equiv0~(\rm{mod}~2)$ when $a=b$, then $G$ has a fractional $[a,b]$-factor
unless $G\cong K_{a-1}\vee(K_{n-a}+K_1)$.
\end{thm}

Note that $4+\sqrt{32a^{2}+24a+5}> a+1$. Our Theorem \ref{main1} improves the following result.

\begin{thm}[Wang, Zheng and Chen \cite{Wang2023}]\label{th24}
Let $b\geq a\geq 1$ be two integers, and let $G$ be a graph of order $n\geq 4+\sqrt{32a^{2}+24a+5}$.
If $\rho(G)\geq\rho(K_{a-1}\vee(K_{n-a}+K_1))$, then $G$ has a fractional $[a,b]$-factor unless $G\cong K_{a-1}\vee(K_{n-a}+K_1)$.
\end{thm}

A graph $G$ is {\it independent-set-deletable factor-critical}, shortly {\it ID-factor-critical}, if for every
independent set $I$ of $G$ whose size has the same parity as $|V(G)|$, $G-I$ has a perfect matching.
Let $S_{n, k}$ be the join of a clique on $k$ vertices with an independent set of $n-k$ vertices for $n>k$. That is to say, $S_{n, k}=K_{k}\vee I_{n-k}$.

\begin{thm}[Tutte \cite{Tutte1947}]
A graph $G$ has a perfect matching if and only if $o(G-S)\leq|S|$ for every $S\subseteq V(G)$.
\end{thm}

The following theorem is a direct consequence of Tutte's Theorem.

\begin{thm}\label{thm100}
A graph $G$ is ID-factor-critical if and only if $o(G-I-S)\leq|S|$ for every independent set $I$ such that $|I|$ has the same parity as $|V(G)|$ and every subset $S\subseteq V(G)-I$.
\end{thm}

Using Theorem \ref{thm100}, we prove a tight spectral condition for a graph with minimum degree $\delta$ to be ID-factor-critical.

\begin{thm}\label{thm101}
Let $G$ be a graph of order $n$ with minimum degree $\delta\geq3r+1$, where $r\geq1$ is an integer.
If $n\geq {\rm max} \left\{20\delta+r+8, \delta^3-\frac{r-3}{2}\delta^2-\frac{r^2-2r-4}{2}\delta-\frac{r^2-3r-3}{2}\right\}$ and
$$\rho(G)\geq\rho(S_{\delta+r, \delta}\vee(K_{n-2\delta-r-1}+I_{\delta+1})),$$
then $G$ is ID-factor-critical unless $G\cong S_{\delta+r, \delta}\vee(K_{n-2\delta-r-1}+I_{\delta+1})$.
\end{thm}

\section{Proof of Theorem \ref{main1}}

Before presenting our proof, we introduce some necessary lemmas.

\begin{lem}[Wei and Zhang \cite{Wei2023}]\label{le1}
Let $G$ be a graph of order $n\geq3$. If $e(G)\geq\binom{n-1}{2}+1$, then $G$ has a Hamilton path.
\end{lem}

Although the following Lemma \ref{le2} can be obtained directly from Theorem 2 in \cite{Wei2023},
here we can present a much simpler proof of Lemma \ref{le2} for a fractional $[a,b]$-factor.

\begin{lem}\label{le2}
Let $a$ and $b$ be two positive integers with $a\leq b$, and let $G$ be a graph of order $n\geq a+1$ and minimum degree $\delta\geq a$. If
$$e(G)\geq\binom{n-1}{2}+\frac{a+1}{2}$$
and $na\equiv0~(\rm{mod}~2)$ when $a=b$, then $G$ has a fractional $[a,b]$-factor.
\end{lem}

\begin{proof}
For any two disjoint vertex subsets $S$ and $T$ in $G$, let
$$\varphi(S,T)=b|S|-a|T|+\sum_{v\in T}d_{G-S}(v).$$
 Suppose to the contrary that $G$ has no fractional $[a,b]$-factor. By Corollary \ref{cor1},
 there exist two disjoint subsets $S$ and $T$ of $V(G)$ such that
\begin{flalign}\label{eq1}
&&\varphi(S,T)\leq -1,&&
\end{flalign}
where $T=\{v|v\in V(G)-S~\mbox{and}~d_{G-S}(v)<a\}$.

\begin{claim}\label{claim1}
$n\geq a+2$ and $b\geq2.$
\end{claim}

\begin{proof}
Note that $\delta\geq a$. If $n=a+1$, then $G$ is a complete graph. It is well known that the complete graph contains an $[a,b]$-factor,
and hence $G$ contains a fractional $[a,b]$-factor, a contradiction. So we have $n\geq a+2$.

If $b=1$, then $a=b=1$, and thus $e(G)\geq\binom{n-1}{2}+1$.
By Lemma \ref{le1}, $G$ has a Hamilton path.
Note that $n$ is even. Then $G$ contains a $1$-factor, and hence $G$ contains a fractional $1$-factor, a contradiction.
Hence $b\geq2.$
\end{proof}

\begin{claim}\label{claim2}
$S\neq\emptyset$
\end{claim}
\begin{proof}
Assume that $S=\emptyset$.
Note that $G-S=G$ and $\delta(G)\geq a$. Then $\delta(G-S)\geq a$. Recall that $T=\{v|v\in V(G)-S~\mbox{and}~d_{G-S}(v)<a\}$.
Then $T=\emptyset$, and thus $$\varphi(\emptyset,\emptyset)=0,$$
which is contrary to (\ref{eq1}).
\end{proof}

Next we will evaluate the value of $|T|$.

\vspace{1.5mm}
\noindent\textbf{Case 1.} $0\leq|T|\leq b$.
\vspace{1mm}

Note that $\delta\geq a$. Then
\begin{eqnarray*}
\varphi(S,T) &=&b|S|-a|T|+\sum_{v\in T}d_{G-S}(v)\\
&=&b|S|-a|T|+\sum_{v\in T}d_{G}(v)-e_G(S,T)\\
&\geq&b|S|-a|T|+a|T|-|T||S|\\
&=&(b-|T|)|S|\\
&\geq&0,
\end{eqnarray*}
which contradicts (\ref{eq1}).

\vspace{1.5mm}
\noindent\textbf{Case 2.} $|T|\geq b+1$.
\vspace{1mm}

Since $S$ and $T$ are two disjoint subsets of $V(G)$, $n\geq|S|+|T|\geq|S|+b+1.$
By the assumption $e(G)\geq\binom{n-1}{2}+\frac{a+1}{2}$, there exist at most $n-1-\frac{a+1}{2}$ edges which are not in $E[V(G-T-S),T]\cup E(G[T])$.
Hence
\begin{eqnarray}
\sum_{v\in T}d_{G-S}(v)\geq(n-1-|S|)|T|-2(n-1-\frac{a+1}{2}).\nonumber
\end{eqnarray}

\vspace{1.5mm}
\noindent\textbf{Subcase 2.1.} $a<b$.
\vspace{1mm}

Combining Claim \ref{claim1}, we have
\begin{eqnarray*}
\varphi(S,T)
&=& b|S|-a|T|+\sum_{v\in T}d_{G-S}(v)\\
&\geq& b|S|-a|T|+(n-1-|S|)|T|-2(n-1-\frac{a+1}{2})\\
&=& (n-1-|S|-a)|T|+b|S|-2n+a+3\\
&\geq& (n-1-|S|-a)(b+1)+b|S|-2n+a+3\\
&=& (b-2)n+n-|S|-ab-b+2\\
&\geq& (b-2)n+(|S|+b+1)-|S|-ab-b+2\\
&=& (b-2)n-ab+3\\
&\geq& (b-2)(a+2)-ab+3\\
&=& 2b-2a-1\\
&\geq& 1,
\end{eqnarray*}
which is contrary to (\ref{eq1}).

\vspace{1.5mm}
\noindent\textbf{Subcase 2.2.} $a=b$.
\vspace{1mm}

Recall that $n\geq|S|+b+1=|S|+a+1$ and $na \equiv0~(\rm{mod}~2)$. If $a$ is odd, then $n$ is even.
By Claim \ref{claim1}, we have $n\geq a+3$ and $a\geq3$.
Then
\begin{eqnarray}
\varphi(S,T)\nonumber
&=&a|S|-a|T|+\sum_{v\in T}d_{G-S}(v)\\ \nonumber
&\geq&a|S|-a|T|+(n-1-|S|)|T|-2(n-1-\frac{a+1}{2})\\ \nonumber
&=&(n-1-|S|-a)|T|+a|S|-2n+a+3\\ \nonumber
&\geq&(n-1-|S|-a)(a+1)+a|S|-2n+a+3\\ \nonumber
&=&(a-2)n+n-|S|-a^2-a+2\\ \nonumber
&\geq&(a-2)n+(|S|+a+1)-|S|-a^2-a+2\\ \nonumber
&=&(a-2)n-a^2+3\\ \nonumber
&\geq&(a-2)(a+3)-a^2+3\\ \nonumber
&=&a-3.\\ \nonumber
&\geq&0,
\end{eqnarray}
a contradiction.

Next we consider that $a$ is even. Since $e(G)\geq\binom{n-1}{2}+\frac{a+1}{2}$, we obtain that $e(G)\geq\binom{n-1}{2}+\frac{a+2}{2}$,
and hence $\sum_{v\in T}d_{G-S}(v)\geq(n-1-|S|)|T|-2(n-1-\frac{a+2}{2}).$
By Claim \ref{claim1}, we have $n\geq a+2$. Then
\begin{eqnarray}
\varphi(S,T)\nonumber
&=& a|S|-a|T|+\sum_{v\in T}d_{G-S}(v)\\ \nonumber
&\geq& a|S|-a|T|+(n-1-|S|)|T|-2(n-1-\frac{a+2}{2})\\ \nonumber
&=& (n-1-|S|-a)|T|+a|S|-2n+a+4\\ \nonumber
&\geq& (n-1-|S|-a)(a+1)+a|S|-2n+a+4\\ \nonumber
&=& (a-2)n+n-|S|-a^2-a+3\\ \nonumber
&\geq& (a-2)n+(|S|+a+1)-|S|-a^2-a+3\\ \nonumber
&=& (a-2)n-a^2+4\\ \nonumber
&\geq& (a-2)(a+2)-a^2+4\\ \nonumber
&=& 0,
\end{eqnarray}
which contradicts (\ref{eq1}).
\end{proof}

Let $A=(a_{ij})$ and $B=(b_{ij})$ be two $n\times n$ matrices.
Define $A\leq B$ if $a_{ij}\leq b_{ij}$ for all $i$ and $j$, and define $A< B$ if $A\leq B$ and $A\neq B$.

\begin{lem}[Berman and Plemmons \cite{Berman1979}, Horn and Johnson \cite{Horn1986}]\label{le4}
Let $A=(a_{ij})$ and $B=(b_{ij})$ be two $n\times n$ matrices with the spectral radii $\lambda(A)$ and $\lambda(B)$.
If $0\leq A\leq B$, then $\lambda(A)\leq \lambda(B)$.
Furthermore, if $B$ is irreducible and $0\leq A < B$, then $\lambda(A)<\lambda(B)$.
\end{lem}

We will use the following lemma in the proof of Theorem \ref{main1}.

\begin{lem}[Hong, Shu and Fang \cite{Hong2001}, Nikiforov \cite{Nikiforov2002}]\label{le3}
Let $G$ be a graph with minimum degree $\delta.$ Then
$$\rho(G)\leq \frac{\delta-1}{2}+\sqrt{2e(G)-\delta n+\frac{(\delta+1)^{2}}{4}}.$$
\end{lem}

\begin{pro}[Hong, Shu and Fang \cite{Hong2001}, Nikiforov \cite{Nikiforov2002}]\label{pro}
For graph $G$ with $2e(G)\leq n(n-1),$ the function
$$f(x)=\frac{x-1}{2}+\sqrt{2e(G)-nx+\frac{(x+1)^{2}}{4}}$$
is decreasing with respect to $x$ for $0\leq x\leq n-1.$
\end{pro}

\medskip
\noindent  \textbf{Proof of Theorem \ref{main1}.} Let $G$ be a graph of order $n\geq a+1$.
Note that the minimum degree of $K_{a-1}\vee (K_{n-a}+K_1)$ is $a-1$.
Let $h: E(G)\rightarrow [0,1]$ be a function. Then for $v\in V(K_1)$, we have $\sum_{e\in E_{G}(v)}h(e)\leq a-1$.
By the definition of a fractional $[a, b]$-factor, then $K_{a-1}\vee (K_{n-a}+K_1)$ has no fractional $[a, b]$-factor.
Assume that $G\ncong K_{a-1}\vee (K_{n-a}+K_1)$ (see Fig. \ref{f1}). It suffices to prove that $G$ contains a fractional $[a, b]$-factor.
First we prove the following claim.

\begin{figure}
\centering
\includegraphics[width=0.40\textwidth]{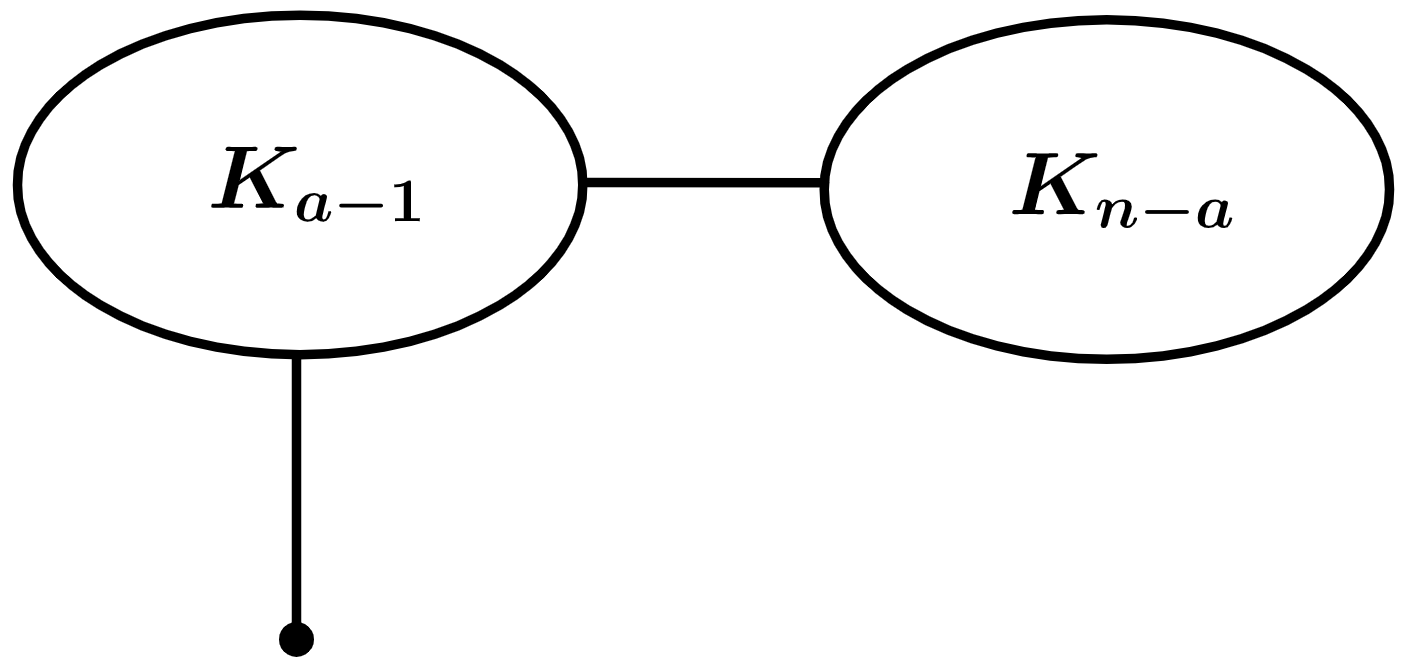}\\
\caption{Graph $K_{a-1}\vee (K_{n-a}+K_{1}).$
}\label{f1}
\end{figure}

\begin{claim}\label{cla33}
$\delta\geq a$.
\end{claim}

\begin{proof}
If $\delta\leq a-1$, then there exists a vertex $v\in V(G)$ such that $d(v)\leq a-1$.
This means that $G$ is a subgraph of $K_{a-1}\vee (K_{n-a}+K_1)$.
By Lemma \ref{le4}, we have $$\rho(G)\leq\rho(K_{a-1}\vee (K_{n-a}+K_1)).$$
By the assumption $\rho(G)\geq\rho(K_{a-1}\vee (K_{n-a}+K_1))$, we have $G\cong K_{a-1}\vee (K_{n-a}+K_1)$, a contradiction.
Hence $\delta\geq a$.
\end{proof}

We distinguish the proof into the following two cases.

\vspace{1.5mm}
\noindent\textbf{Case 1.} $a=1$.
\vspace{1mm}

By the assumption, we have $\rho(G)\geq \rho(K_{a-1}\vee (K_{n-a}+K_1))= \rho(K_{n-1}+K_1)=n-2$.
By Claim \ref{cla33}, Lemma \ref{le3} and Proposition \ref{pro}, we obtain that
$$n-2\leq\rho(G)\leq \sqrt{2e(G)-n+1}.$$
It follows that $e(G)\geq\binom{n-1}{2}+\frac{1}{2}$, and hence $e(G)\geq\binom{n-1}{2}+1$.
By Lemma \ref{le2}, then $G$ contains a fractional $[a,b]$-factor.

\vspace{1.5mm}
\noindent\textbf{Case 2.} $a\geq 2$.
\vspace{1mm}

Note that $K_{n-1}$ is a proper subgraph of $K_{a-1}\vee (K_{n-a}+K_1)$.
By the assumption and Lemma \ref{le4}, we have $\rho(G)\geq\rho(K_{a-1}\vee (K_{n-a}+K_1))>\rho(K_{n-1})=n-2$.
By Claim \ref{cla33}, Lemma \ref{le3} and Proposition \ref{pro}, we have
$$n-2<\rho(G)\leq\frac{a-1}{2}+\sqrt{2e(G)-an+\frac{(a+1)^{2}}{4}}.$$
It follows that $e(G)>\binom{n-1}{2}+\frac{a}{2}$. That is to say, $e(G)\geq\binom{n-1}{2}+\frac{a+1}{2}$.
By Lemma \ref{le2}, then $G$ contains a fractional $[a,b]$-factor. \hspace*{\fill}$\Box$

\section{Proof of Theorem \ref{thm101}}

By the Perron-Frobenius Theorem, $\rho(G)$ is always a positive number (unless $G$ is an empty graph),
and there exists an unique positive unit eigenvector corresponding to $\rho(G)$, which is called the {\it Perron vector} of $G$.

\begin{lem}[Fan, Goryainov, Huang and Lin\cite{Fan2021}, Fan and Lin\cite{Fan}]\label{le33}
Let $n=\sum_{i=1}^{t}n_i+s.$ If $n_{1}\geq n_{2}\geq \cdots \geq n_{t}\geq p $ and $n_{1}<n-s-p(t-1)$, then
$$\rho(K_{s}\vee(K_{n_{1}}+ K_{n_{2}} + \cdots + K_{n_{t}}))<\rho(K_{s}\vee(K_{n-s-p(t-1)}+ (t-1)K_{p})).$$
\end{lem}

\begin{lem}\label{le44}
Graph $S_{\delta+r,\delta}\vee(K_{n-2\delta-r-1}+I_{\delta+1})$ is not ID-factor-critical.
\end{lem}
\begin{proof}
Let $G= S_{\delta+r,\delta}\vee(K_{n-2\delta-r-1}+I_{\delta+1})$ (see Fig. \ref{f2}).
Suppose to the contrary that $G$ is ID-factor-critical.
By the definition of an ID-factor-critical graph,
we have $G-I$ has a perfect matching for any independent set $I$ of $G$ whose size has the same parity as $|V(G)|$.

Note that $S_{\delta+r, \delta}=K_{\delta}\vee I_{r}$. However, if we take $I=I_r$ and let $H=G-I$.
Then $$H\cong K_{\delta}\vee(K_{n-2\delta-r-1}+I_{\delta+1}).$$
Note that the vertices of $I_{\delta+1}$ are only adjacent to the vertices of $K_{\delta}$. Hence $H$ has no perfect matching, a contradiction.
\end{proof}
Now, we are in a position to present the proof of Theorem \ref{thm101}.
\begin{figure}
\centering
\includegraphics[width=0.40\textwidth]{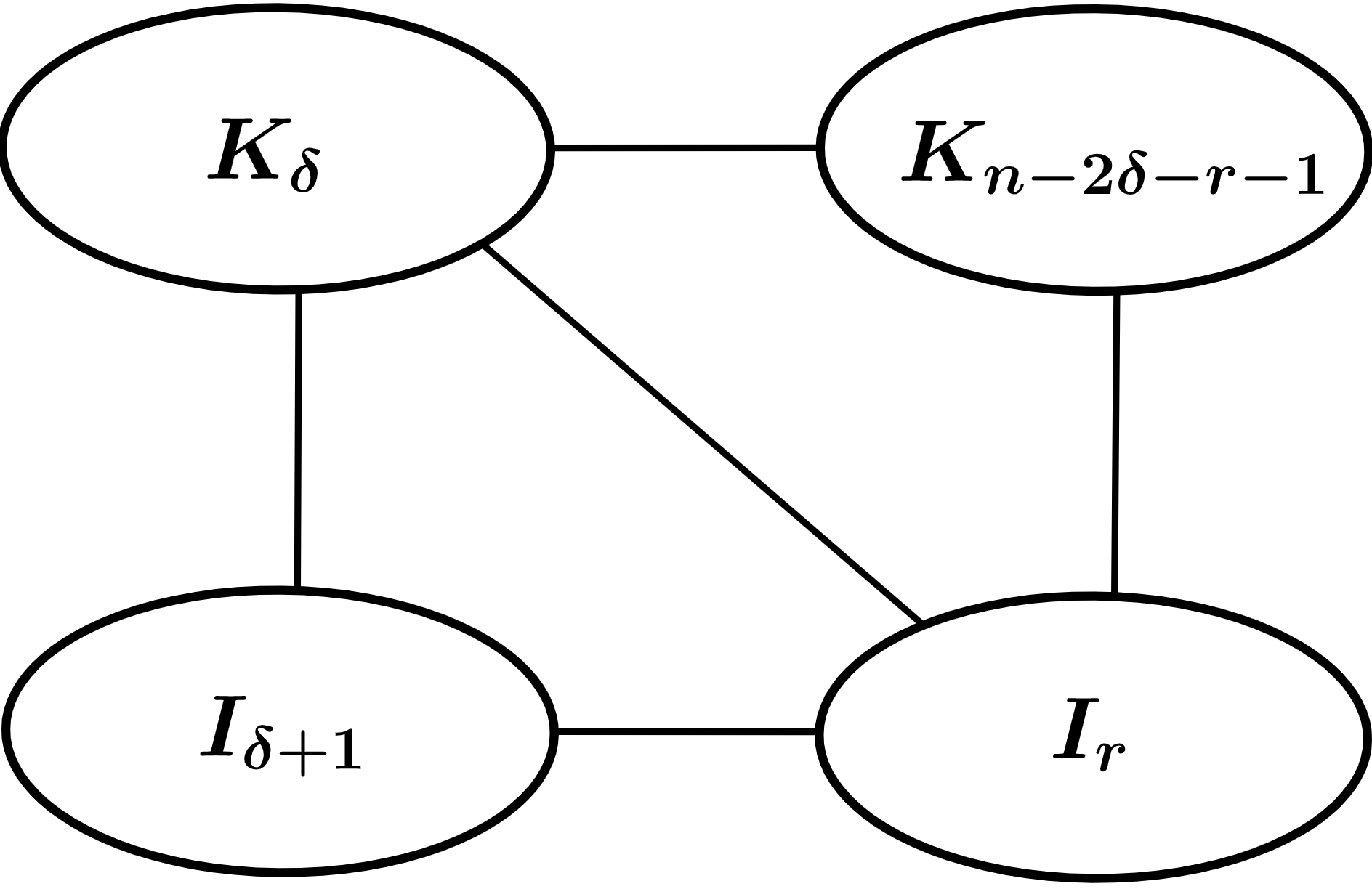}\\
\caption{Graph $S_{\delta+r,\delta}\vee(K_{n-2\delta-r-1}+I_{\delta+1}).$
}\label{f2}
\end{figure}

\medskip
\noindent  \textbf{Proof of Theorem \ref{thm101}.}
Assume that $G$ is not ID-factor-critical. According to Theorem \ref{thm100}, there exists some independent set $I$ such that $|I|$ has the same parity as
$|V(G)|=n$ , we have $o(G-I-S)\geq |S|+1$ for some subset $S\subseteq V(G)-I$. Let $|I|=r$ and $|S|=s$. Then $o(G-I-S)\geq s+1$.
Since $n-r$ is even, $o(G-I-S)$ and $s$ have the same parity. Hence we have $o(G-I-S)\geq s+2$.
It is clear that $G$ is a spanning subgraph of $G'=I_r\vee (K_s\vee(K_{n_{1}}+ K_{n_{2}}+\cdots+K_{n_{s+2}}))$
for some odd integers $n_{1}\geq n_{2}\geq \cdots \geq n_{s+2}>0$ with $\sum_{i=1}^{s+2}n_{i}=n-r-s$. Then we have
\begin{eqnarray}\label{equ1}
\rho(G)\leq\rho(G'),
\end{eqnarray}
where equality holds if and only if $G\cong G'$.
Let $G''=S_{s+r, s}\vee(K_{n-2s-r-1}+I_{s+1})$. By Lemma \ref{le33}, we obtain that
\begin{eqnarray}\label{equ2}
\rho(G')\leq \rho(G''),
\end{eqnarray}
where equality holds if and only if $(n_{1}, n_{2}, \ldots ,n_{s+2})=(n-2s-r-1,1,\ldots ,1)$.

\vspace{1.5mm}
\noindent\textbf{Case 1.} $s=\delta$.
\vspace{1mm}

Combining (\ref{equ1}) and (\ref{equ2}), we have
$$\rho(G)\leq\rho(G')\leq \rho(G'')=\rho(S_{\delta+r, \delta}\vee(K_{n-2\delta-r-1}+I_{\delta+1})).$$
By the assumption $\rho(G)\geq\rho(S_{\delta+r, \delta}\vee(K_{n-2\delta-r-1}+I_{\delta+1}))$,
we have $G\cong S_{\delta+r, \delta}\vee(K_{n-2\delta-r-1}+I_{\delta+1})$.
By Lemma \ref{le44}, $S_{\delta+r, \delta}\vee(K_{n-2\delta-r-1}+I_{\delta+1})$ is not ID-factor-critical.
Hence $G\cong S_{\delta+r, \delta}\vee(K_{n-2\delta-r-1}+I_{\delta+1})$.

\vspace{1.5mm}
\noindent\textbf{Case 2.} $s\geq \delta+1$.
\vspace{1mm}

Recall that $G''=S_{s+r, s}\vee(K_{n-2s-r-1}+I_{s+1})$.
The vertex set of $G''$ can be divided into $V(G'')=V(K_{s})\cup V(I_{s+1})\cup V(I_{r})\cup V(K_{n-2s-r-1})$,
where $V(K_{s})=\left\{u_{1}, u_{2}, \ldots ,u_{s} \right\}$,
$V(I_{s+1})=\left \{v_{1}, v_{2}, \ldots ,v_{s+1}\right\}$, $V(I_r)=\left \{w_{1}, w_{2}, \ldots ,w_{r}\right\}$
and $V(K_{n-2s-r-1})=\left\{z_{1}, z_{2}, \ldots ,z_{n-2s-r-1}\right\}$.
Let $E_{1}=\left\{v_{i}z_{j}|\delta+2\leq i\leq s+1, 1\leq j\leq n-2s-r-1\right\}$
$\cup\left\{v_iv_j|\delta+2\leq i\leq s,i+1\leq j\leq s+1\right\}$
and $E_{2}=\left\{u_{i}v_{j}|\delta+1\leq i\leq s,1\leq j\leq\delta+1\right\}$.
Let $G^*=G''+E_{1}-E_{2}$. Obviously, $G^*\cong S_{\delta+r, \delta}\vee(K_{n-2\delta-r-1}+I_{\delta+1}).$
Let $x$ be the perron vector of $A(G'')$, and let $\rho'' =\rho(G'')$.
By symmetry, $x$ takes the same value on the vertices of $V(K_{s})$, $V(I_{s+1})$, $V(I_r)$ and $V(K_{n-2s-r-1})$, respectively.
It is easy to see that
$$A(G'')=\left[
\begin{array}{cccc}
(J-I)_{s\times s}&J_{s\times (s+1)}&J_{s\times r}&J_{s\times (n-2s-r-1)}\\
J_{(s+1)\times s}&O_{(s+1)\times (s+1)}&J_{(s+1)\times r}&O_{(s+1)\times (n-2s-r-1)}\\
J_{r\times s}&J_{r\times (s+1)}&O_{r\times r}&J_{r\times(n-2s-r-1)}\\
J_{(n-2s-r-1)\times s}&O_{(n-2s-r-1)\times (s+1)}&J_{(n-2s-r-1)\times r}&(J-I)_{(n-2s-r-1)\times(n-2s-r-1)}
\end{array}
\right].
$$
We denote the entry of $x$ by $x_{1}$, $x_{2}$, $x_{3}$ and $x_4$ corresponding to the vertices in the above four vertex sets, respectively.
By $A(G'')x=\rho'' x$, we have
\begin{eqnarray} \label{equ100}
&&\rho'' x_{2}=sx_{1}+rx_3, \\ \label{equ101}
&&\rho'' x_{3}=sx_{1}+(s+1)x_2+(n-2s-r-1)x_4,   \\ \label{equ102}
&&\rho'' x_{4}=sx_{1}+rx_3+(n-2s-r-2)x_4.
\end{eqnarray}
Observe that $n\geq2s+r+2$. According to (\ref{equ100}) and (\ref{equ102}), we obtain that $x_{4}\geq x_{2}$.
By (\ref{equ101}) and (\ref{equ102}), we have $\rho'' x_3-\rho'' x_{4}=(s+1)x_{2}-rx_3+x_4$.
It follows that $x_4=\frac{(\rho''+r)x_3-(s+1)x_2}{\rho''+1}\geq x_2$.
Then we have $x_3\geq \frac{\rho''+s+2}{\rho''+r}x_2$.
Note that $s\geq\delta+1$ and $\delta\geq3r+1$. Then $\rho''+s+2\geq\rho''+\delta+3>\rho''+r$, and hence $x_3>x_2$.
Combining (\ref{equ100}), we have
\begin{eqnarray} \label{equ103}
x_2>\frac{sx_1}{\rho''-r}.
\end{eqnarray}
Recall that $G^*\cong S_{\delta+r, \delta}\vee(K_{n-2\delta-r-1}+I_{\delta+1}).$
Note that $G^*$ contains $K_{n-2\delta-r-1}$ as a proper subgraph. Then $\rho^*>n-2\delta-r-2$.
Similarly, let $y$ be the perron vector of $A(G^*)$, and let $\rho^* =\rho(G^*)$.
By symmetry, $y$ takes the same value (say $y_{1}$, $y_{2}$, $y_{3}$ and $y_{4}$) on the vertices of
$V(K_{\delta})$, $V(I_{\delta+1})$, $V(I_r)$ and $V(K_{n-2\delta-r-1})$.
By $A(G^*)y=\rho^*y$, we have
\begin{eqnarray}\label{equ104}
&&\rho^*y_{2}=\delta y_{1}+ry_3, \\ \label{equ106}
&&\rho^*y_{4}=\delta y_{1}+ry_{3}+(n-2\delta-r-2)y_4.
\end{eqnarray}
Combining (\ref{equ104}) and (\ref{equ106}), we have
\begin{eqnarray}\label{equ107}
y_{4}=\frac{\rho^*y_{2}}{\rho^*-(n-2\delta-r-2)}.
\end{eqnarray}
Note that $n\geq2s+r+2$. Then $\delta+1\leq s\leq \frac{n-r-2}{2}$. Since $G''$ is not a complete graph, $\rho''< n-1$.

\begin{claim}\label{claim100}
 $\rho''<\rho^*$.
\end{claim}
\begin{proof}
Suppose that $\rho''\geq \rho^*$. By $x_4\geq x_2$, (\ref{equ103}) and (\ref{equ107}), we have
\begin{eqnarray*}
&&y^{T}(\rho^*-\rho'')x\\
&=&y^{T}(A(G^*)-A(G''))x\\
&=&\sum_{i=\delta+2}^{s+1}\sum_{j=1}^{n-2s-r-1}(x_{v_{i}}y_{z_j}+x_{z_j}y_{v_i})+\sum_{i=\delta+2}^{s}
\sum_{j=i+1}^{s+1}(x_{v_{i}}y_{v_{j}}+x_{v_{j}}y_{v_{i}})-\sum_{i=\delta+1}^{s}\sum_{j=1}^{\delta+1}(x_{u_{i}}y_{v_{j}}+x_{v_{j}}y_{u_{i}})\\
&=&(n-2s-r-1)(s-\delta)(x_{2}y_{4}+x_{4}y_{4})+(s-\delta-1)(s-\delta)x_{2}y_{4}-(s-\delta)(\delta+1)(x_{1}y_{2}\\
&&+x_{2}y_{4})\\
&\geq&(s-\delta)\left[2(n-2s-r-1)x_2y_4+(s-\delta-1)x_2y_{4}-(\delta+1)x_{2}y_{4}-(\delta+1)x_1y_2\right]\\
&=&(s-\delta)\left[(2n-3s-2\delta-2r-4)x_{2}y_{4}-(\delta+1)x_{1}y_{2}\right]\\
&>&(s-\delta)\left[(2n-3s-2\delta-2r-4)\cdot\frac{sx_1}{\rho''-r}\cdot\frac{\rho^*y_2}{\rho^*-(n-2\delta-r-2)}-(\delta+1)x_{1}y_{2}\right]\\
&=&\frac{(s-\delta)x_1y_2}{(\rho''-r)(\rho^*-(n-2\delta-r-2))}\left[(2n-3s-2\delta-2r-4)s\rho^*-(\delta+1)(\rho''-r)(\rho^* \right. \\
&& \left.-(n-2\delta-r-2))\right]\\
&=&\frac{(s-\delta)(\delta+1)x_1y_2}{(\rho''-r)(\rho^*-(n-2\delta-r-2))}\left[\rho^*(2n-2\delta-3s-2r-4)\cdot\frac{s}{\delta+1}-(\rho''-r)(\rho^*\right. \\
&& \left.-(n-2\delta-r-2))\right].
\end{eqnarray*}
Note that $s\geq\delta+1$, $\rho''\geq\rho^*$ and $\rho^*>\delta-1\geq3r$. Then
\begin{eqnarray*}
&&y^{T}(\rho^*-\rho'')x\\
&>&\frac{(s-\delta)(\delta+1)x_1y_2}{(\rho''-r)(\rho^*-(n-2\delta-r-2))}\left[\rho^*(2n-2\delta-3s-2r-4)-\rho''\rho^*+\rho''(n-2\delta-r-2)\right.\\
&&\left.+r\rho^*-r(n-2\delta-r-2)\right]\\
&=&\frac{\rho^*(s-\delta)(\delta+1)x_1y_2}{(\rho''-r)(\rho^*-(n-2\delta-r-2))}\left[(2n-2\delta-3s-2r-4)-\rho''+\frac{\rho''}{\rho^*}\cdot(n-2\delta-r-2)\right.\\
&&\left.+r-\frac{r}{\rho^*}\cdot(n-2\delta-r-2)\right]\\
&>&\frac{\rho^*(s-\delta)(\delta+1)x_1y_2}{(\rho''-r)(\rho^*-(n-2\delta-r-2))}\left[(2n-2\delta-3s-2r-4)-\rho''+(n-2\delta-r-2)\right.\\
&&\left.+r-\frac{1}{3}\cdot(n-2\delta-r-2)\right]\\
&=&\frac{\rho^*(s-\delta)(\delta+1)x_1y_2}{(\rho''-r)(\rho^*-(n-2\delta-r-2))}\left(\frac{8}{3}n-\frac{10}{3}\delta-3s-\frac{5}{3}r-\frac{16}{3}-\rho''\right).\\
\end{eqnarray*}
Since $K_s\subset G''$ and $\delta\geq3r+1$, $\rho''> \rho(K_s)=s-1\geq \delta >r$.
Note that $s\leq\frac{n-r-2}{2}$, $\rho''< n-1$, $\rho^*>n-2\delta-r-2$ and $n\geq 20\delta+r+8$. Then
\begin{eqnarray*}
&&y^{T}(\rho^*-\rho'')x\\
&>&\frac{\rho^*(s-\delta)(\delta+1)x_1y_2}{(\rho''-r)(\rho^*-(n-2\delta-r-2))}\left(\frac{8}{3}n-\frac{10}{3}\delta-3\cdot\frac{n-r-2}{2}-\frac{5}{3}r-\frac{16}{3}-\rho''\right)\\
\end{eqnarray*}
\begin{eqnarray*}
&=&\frac{\rho^*(s-\delta)(\delta+1)x_1y_2}{(\rho''-r)(\rho^*-(n-2\delta-r-2))}\left(\frac{7}{6}n-\frac{10}{3}\delta-\frac{1}{6}r-\frac{7}{3}-\rho''\right)\\
&=&\frac{\rho^*(s-\delta)(\delta+1)x_1y_2}{(\rho''-r)(\rho^*-(n-2\delta-r-2))}\left(\frac{1}{6}n-\frac{10}{3}\delta-\frac{1}{6}r-\frac{4}{3}+(n-1)-\rho''\right)\\
&>&\frac{\rho^*(s-\delta)(\delta+1)x_1y_2}{(\rho''-r)(\rho^*-(n-2\delta-r-2))}\cdot\frac{n-20\delta-r-8}{6}\\
&\geq&0.
\end{eqnarray*}
This implies that $\rho^*>\rho''$, which contradicts the assumption $\rho''\geq \rho^*$.
\end{proof}
By Claim \ref{claim100}, (\ref{equ1}) and (\ref{equ2}), we have
$$\rho(G)\leq\rho(G')\leq\rho(G'')<\rho(G^*)=\rho(S_{\delta+r, \delta}\vee(K_{n-2\delta-r-1}+I_{\delta+1})),$$
which contradicts $\rho(G)\geq\rho(S_{\delta+r, \delta}\vee(K_{n-2\delta-r-1}+I_{\delta+1}))$.

\vspace{1.5mm}
\noindent\textbf{Case 3.}  $s<\delta.$
\vspace{1mm}

Recall that $G'=I_r\vee (K_s\vee(K_{n_{1}}+ K_{n_{2}}+\cdots+K_{n_{s+2}}))$.
Then $d_{G'}(v)=n_{s+2}+s+r-1\geq \delta$ for $v \in V(K_{n_{s+2}})$, and hence $n_{s+2}\geq\delta-s-r+1$.
Let $G'''=I_r\vee(K_s\vee(K_{n-s-r-(s+1)(\delta-s-r+1)}+(s+1)K_{\delta-s-r+1}))$.
By Lemma \ref{le33}, we have
\begin{eqnarray}\label{equ7}
\rho(G')\leq\rho(G'''),
\end{eqnarray}
where equality holds if and only if $(n_1,n_2,\ldots,n_{s+2})=(n-s-r-(s+1)(\delta-s-r+1),\delta-s-r+1,\ldots,\delta-s-r+1).$
Let $\rho'''=\rho(G''')$.

\begin{claim}\label{claim101}
$\rho'''< n-r-1-(s+1)(\delta-s+1)$.
\end{claim}
\begin{proof}
Suppose to the contrary that  $\rho'''\geq n-r-1-(s+1)(\delta-s+1)$.
Let $x$ be the perron vector of $A(G''')$. By symmetry, $x$ takes the same values $x_1$, $x_2$, $x_3$ and $x_4$
on the vertices of $K_s$, $(s+1)K_{\delta-s-r+1}$, $I_r$ and $K_{n-s-r-(s+1)(\delta-s-r+1)}$, respectively.
According to $A(G''')x=\rho''' x$, we obtain that
\begin{eqnarray}
&&\rho'''x_1=(s-1)x_1+(s+1)(\delta-s-r+1)x_2+rx_3+(n-s-r-(s+1)(\delta-s \nonumber \\ \label{equ8}
&&~~~~~~~~~~~~~~-r+1))x_4,\\ \label{equ9}
&&\rho'''x_2=sx_1+(\delta-s-r)x_2+rx_3,\\ \label{equ10}
&&\rho'''x_3=sx_1+(s+1)(\delta-s-r+1)x_2+(n-s-r-(s+1)(\delta-s-r+1))x_4,\\ \label{equ1000}
&&\rho'''x_4=sx_1+rx_3+(n-s-r-1-(s+1)(\delta-s-r+1))x_4.
\end{eqnarray}
By (\ref{equ8}) and (\ref{equ10}), we have
\begin{eqnarray}\label{equ13}
&&x_3=\frac{(\rho'''+1)x_1}{\rho'''+r}.
\end{eqnarray}
Substituting (\ref{equ13}) into (\ref{equ9}) and (\ref{equ1000}), we have
\begin{eqnarray}\label{equ11}
&&x_2=\frac{sx_1+\frac{r(\rho'''+1)}{\rho'''+r}x_1}{\rho'''-\delta+s+r},\\\label{equ12}
&&x_4=\frac{sx_1+\frac{r(\rho'''+1)}{\rho'''+r}x_1}{\rho'''-[n-s-r-1-(s+1)(\delta-s-r+1)]}.
\end{eqnarray}
Since $n\geq \delta^3-\frac{r-3}{2}\delta^2-\frac{r^2-2r-4}{2}\delta-\frac{r^2-3r-3}{2}$, we have
$\rho'''\geq n-r-1-(s+1)(\delta-s+1)>\delta-r+1$.
Substituting (\ref{equ13}), (\ref{equ11}) and (\ref{equ12}) into (\ref{equ8}), we have

\begin{eqnarray*}
\rho'''+1&=&s+\frac{(s+1)(\delta-s-r+1)(s+\frac{r(\rho'''+1)}{\rho'''+r})}{\rho'''-\delta+s+r}+\frac{r(\rho'''+1)}{\rho'''+r}\\
&&+\frac{[n-s-r-(s+1)(\delta-s-r+1)](s+\frac{r(\rho'''+1)}{\rho'''+r})}{\rho'''-(n-s-r-1-(s+1)(\delta-s-r+1))}\\
&\leq&s+\frac{(s+1)(\delta-s-r+1)(s+r)}{\rho'''-\delta+s+r}+r\\
&&+\frac{[n-s-r-(s+1)(\delta-s-r+1)](s+r)}{\rho'''-(n-s-r-1-(s+1)(\delta-s-r+1))}\\
&<&s+\frac{(s+1)(\delta-s-r+1)(s+r)}{(\delta-r+1)-\delta+s+r}+r\\
&&+\frac{[n-s-r-(s+1)(\delta-s-r+1)](s+r)}{[n-r-1-(s+1)(\delta-s+1)]-(n-s-r-1-(s+1)(\delta-s-r+1))}\\
&=&s+(s+r)(\delta-s-r+1)+r+\frac{[n-s-r-(s+1)(\delta-s-r+1)](s+r)}{s-sr-r}\\
&=&n-r-1-(s+1)(\delta-s+1)-\frac{1}{{sr-s+r}}[(sr+2r)n+(2r-1)s^3+(2r^2-2\delta r\\
&&+\delta+1)s^2+(r^3-\delta r^2-r^2-3\delta r-2r+1)s+r^3-\delta r^2-3r^2-2\delta r-3r].
\end{eqnarray*}
Let $f(n)=(sr+2r)n+(2r-1)s^3+(2r^2-2\delta r+\delta+1)s^2+(r^3-\delta r^2-r^2-3\delta r-2r+1)s+r^3-\delta r^2-3r^2-2\delta r-3r$.
We assert that $f(n)\geq 0$. Suppose that $f(n)<0$. Then
$n<\frac{1}{sr+2r}[(-2r+1)s^3+(-2r^2+2\delta r-\delta-1)s^2+(-r^3+\delta r^2+r^2+3\delta r+2r-1)s-r^3+\delta r^2+3r^2+2\delta r+3r]$.
Note that $0\leq s<\delta$, $-2r+1<0$, $-2r^2+2\delta r-\delta-1>0$ and $-r^3+\delta r^2+r^2+3\delta r+2r-1>0$. Then
\begin{eqnarray*}
n&<&\frac{1}{sr+2r}[(-2r+1)s^3+(-2r^2+2\delta r-\delta-1)s^2+(-r^3+\delta r^2+r^2+3\delta r+2r-1)s\\
&&-r^3+\delta r^2+3r^2+2\delta r+3r]\\
&<&\frac{1}{2r}[(-2r^2+2\delta r-\delta-1)\delta^2+(-r^3+\delta r^2+r^2+3\delta r+2r-1)\delta-r^3+\delta r^2+3r^2\\
&&+2\delta r+3r]\\
&=&\frac{1}{2r}[(2r-1)\delta^3+(-r^2+3r-1)\delta^2+(-r^3+2r^2+4r-1)\delta-r^3+3r^2+3r]\\
&<&\frac{1}{2r}[2r\delta^3+(-r^2+3r)\delta^2+(-r^3+2r^2+4r)\delta-r^3+3r^2+3r]\\
&=&\delta^3-\frac{r-3}{2}\delta^2-\frac{r^2-2r-4}{2}\delta-\frac{r^2-3r-3}{2},
\end{eqnarray*}
which contradicts $n\geq \delta^3-\frac{r-3}{2}\delta^2-\frac{r^2-2r-4}{2}\delta-\frac{r^2-3r-3}{2}$. Hence $f(n)\geq 0$.
Then
\begin{eqnarray*}
\rho'''+1 &<&n-r-1-(s+1)(\delta-s+1)-\frac{1}{{sr-s+r}}f(n)\\
&<&n-r-1-(s+1)(\delta-s+1)\\
&\leq& \rho''',
\end{eqnarray*}
a contradiction. Therefore, we have $\rho'''<n-r-1-(s+1)(\delta-s+1)$.
\end{proof}

By Claim \ref{claim101} and $s< \delta$, we obtain that
\begin{eqnarray*}
\rho'''&<&n-r-1-(s+1)(\delta-s+1)\\
&=&n-\delta-r-1-[(\delta-s)s+1]\\
&<&n-\delta-r-1.
\end{eqnarray*}
Note that $K_{n-\delta-r}\subset S_{\delta+r, \delta}\vee(K_{n-2\delta-r-1}+I_{\delta+1})$. Then
$$n-\delta-r-1=\rho(K_{n-\delta-r})<\rho(S_{\delta+r, \delta}\vee(K_{n-2\delta-r-1}+I_{\delta+1})).$$
Combining (\ref{equ1}) and (\ref{equ7}), we have
$$\rho(G)\leq\rho(G')\leq\rho(G''')<n-\delta-r-1<\rho(S_{\delta+r, \delta}\vee(K_{n-2\delta-r-1}+I_{\delta+1})),$$
which contradicts the assumption, as desired.  \hspace*{\fill}$\Box$

\vspace{5mm}
\noindent
{\bf Declaration of competing interest}
\vspace{3mm}

The authors declare that they have no known competing financial interests or personal relationships that could have appeared to influence the work reported in this paper.



\end{document}